\documentclass[11pt]{amsart}
\usepackage{amssymb} 
\usepackage[mathscr]{eucal} 
\usepackage{xypic}   
\usepackage{amsmath} 
\usepackage[all, 2cell]{xy}
\usepackage{epsfig}
\usepackage{amscd}
\usepackage{mathrsfs}
\usepackage{tikz-cd}

\def\A{{\mathbb A}}
\def\Q{{\mathbb Q}}

\def\C{{\mathbb C}}
\def\P{{\mathbb P}}

\newcommand*{\longhookrightarrow}{\ensuremath{\lhook\joinrel\relbar\joinrel\rightarrow}}

\newcommand{\Bb}{{\mathscr B}}

\newcommand{\Spec}{{\rm Spec}}

\newcommand{\Bun}{{\Bb un}}

\DeclareFontFamily{OT1}{pzc}{}
\DeclareFontShape{OT1}{pzc}{m}{it}{<->s*[1.21]pzcmi7t}{}
\DeclareMathAlphabet{\mathpzc}{OT1}{pzc}{m}{it}

\def\Spec{\mathrm{Spec}}

\def\Pic{\mathrm{Pic}}
\def\Jac{\mathrm{Jac}}

\def\Pic{\mathrm{Pic}}

\def\L{{\mathcal L}}

\theoremstyle{definition}
\newtheorem{theorem}{Theorem}[section]

\newtheorem{lemma}[theorem]{Lemma}

\newtheorem{definition}[theorem]{Definition}

\newtheorem{example}[theorem]{Example}

\theoremstyle{remark}
\newtheorem{remark}[theorem]{Remark}

\numberwithin{equation}{section}

\begin{document}

\title[Cohomology of moduli stacks of  principal $\C^*$-bundles]
{Cohomology of moduli stacks of principal $\C^*$-bundles over nodal algebraic curves}

\author[Abel Castorena]
{Abel Castorena}
\address{Centro de Ciencias Matem\'aticas
\\Universidad Nacional Aut\'onoma de M\'exico (UNAM) Campus Morelia
\\Morelia, Michoac\'an 
C.P. 58089, , M\'exico
}
\email{abel@matmor.unam.mx}

\author[Frank Neumann]
{Frank Neumann}
\address{Dipartimento di Matematica 'Felice Casorati'\\
Universit\`a di Pavia\\
via Ferrata 5, 27100 Pavia, Italy}
\email{frank.neumann@unipv.it}

\subjclass{Primary 14F35, secondary 14H10}

\keywords{algebraic stacks, moduli of principle bundles, line bundles, nodal algebraic curves, gerbes}

\begin{abstract}
{We study moduli stacks of principal $\C^*$-bundles over nodal complex algebraic curves and determine their rational cohomology algebras in terms of Chern classes.}
\end{abstract}
\maketitle
 

\section*{Introduction}

\noindent The moduli stack of principal $G$-bundles over a smooth algebraic curve and its cohomology is fundamentally important in algebraic geometry, number theory and mathematical physics for example in gauge theory, conformal field theory and the geometric Langlands correspondence. The fundamental work of Atiyah and Bott \cite{AB} implies for example, that for any semisimple complex reductive algebraic group $G$ the rational cohomology algebra of the moduli stack of principal $G$-bundles on a fixed smooth complex projective algebraic curve $X$ is freely generated by the K\"unneth components of the Chern classes of the universal principal $G$-bundle over the moduli stack. This is basically a translation in terms of algebraic stacks and stack cohomology of the classical results obtained by Atiyah and Bott, who originally used topological tools like equivariant cohomology and Morse theory. Such a stacky approach towards the cohomology is for example illustrated by Teleman \cite{T}, Heinloth and Schmitt \cite{HeSch} and Gaitsgory and Lurie \cite{GL}. While for a given smooth algebraic curve these moduli stacks have a good behaviour and the cohomology rings can be calculated, the situation changes if the curve has singularities and only partial results are known, for example, when working with moduli spaces of principal bundles over nodal curves. Recently, Basu, Dan and Kaur (see \cite{BaDaKa, DaKa}) showed that the rational cohomology ring of the moduli space of stable locally-free sheaves of rank $2$ and degree $d$ with fixed determinant over an irreducible, nodal curve with exactly one node, of genus $g\geq 2$ is finitely generated and the generators are natural degenerations of the generators obtained by Newstead \cite{Ne} in the  case of a fixed smooth algebraic curve. 

It is therefore natural to ask if similar calculations could be done if one works with moduli stacks instead of coarse moduli spaces. In particular, one could ask what is the cohomology of moduli stacks of principal bundles over a fixed nodal or stable algebraic curve. 

In this note we will just illustrate the simplest case, namely that of the moduli stack of principal $\C^*$-bundles or line bundles over a given nodal algebraic curves and determine its cohomology. More precisely, we determine the rational cohomology algebra of the the moduli stack of principal $\C^*$-bundles over a fixed nodal algebraic curve of compact type in terms of Chern classes. Algebraic invariants of moduli stacks of principal bundles on nodal algebraic curves were also discussed by Bhosle \cite{Bh}, where she determines the Picard groups of these stacks and by Frenkel, Teleman and Tolland \cite{FrTeTo} in their construction of algebro-geometric Gromov-Witten invariants for the classifying stack $\Bb \C^*$ of principal $\C^*$-bundles.

In a following-up article we will study the cohomology of moduli stacks of more general principal $G$-bundles over nodal algebraic curves.

\section{Moduli stacks of principal $\C^*$-bundles over nodal curves and their cohomology algebras}

\noindent We will now study moduli stacks of principal $\C^*$-bundles or line bundles over nodal algebraic curves and determine their rational cohomology algebras. For the general theory of algebraic stacks and their cohomology we refer to \cite{He, LMB, N} and in particular concerning moduli stack of principal bundles also to \cite{GL, HeSch, So}.

Let us recall some basic constructions and properties of nodal algebraic curves. Let $Y$ be a connected reduced nodal algebraic curve with only ordinary nodes as singularities defined over $\Spec(k)$ for an algebraically closed field $k$. We will normally assume that $k=\C$ is the field of complex numbers. We recall that the dual graph $\Gamma_Y$ of the nodal curve $Y$ is a simplicial complex of dimension at most one, defined to have one vertex (a zero-dimensional simplex) for every irreducible component of $Y$, and one edge (a 1-dimensional simplex) connecting two vertices for every node in which the two corresponding components intersect. So, we assume $\Gamma_Y$ has 
$\gamma_Y$ vertices, $\delta_Y$ edges and among the edges there is a loop for every node lying on a single irreducible component of $Y$. If we denote by $c$ the number of connected components of $Y$ and assume that $\Gamma_Y$ has an orientation, then the first Betti number of $Y$ is given as 
$b_1(\Gamma_Y)=\text{dim}_{\Bbb Z}(H_1(\Gamma_Y,\Bbb Z))=\delta_Y-\gamma_Y+c$.
\vskip2mm

\begin{definition} Let $Y$ be a connected reduced nodal algebraic curve with only ordinary nodes as singularities defined over an algebraically closed field $k$. Let $p\in Y$ be a node of $Y$ and $Y_p\to Y$ the normalization of $Y$ at $p$. We say that $p$ is {\em separating} if the number of connected components of $Y_p$ is greater than the number of connected components of $Y$. The nodal curve $Y$ is of {\em compact type} if all of its nodes are separating. We denote by $\delta_Y$ the number of nodes of $Y$ and by $\gamma_Y$ the number of irreducible components $Y_i$ of $Y$. 
\end{definition}

\begin{remark} Let $Y$ be a connected reduced nodal algebraic curve. $Y$ is of compact type if and only if $\Gamma_Y$ is a tree or if and only if $b_1(\Gamma_Y)=0$. For a nodal curve $Y$ of compact type its Jacobian is compact and each irreducible component of $Y$ is a smooth algebraic curve.
\end{remark}

\noindent Let $Y$ be such a nodal algebraic curve of compact type and let $v:C\to Y$ the normalization of $Y$. We have a short exact sequence
$$1\to\mathcal O^*_Y\to v_*(\mathcal O^*_C)\to\mathcal G\to 1,$$
\noindent where $\mathcal G$ is a skycraper sheaf supported on the singular points of $Y$. The associated long cohomology exact sequence to the above exact sequence is then given as:

\begin{multline*} 
1\to H^0(Y,\mathcal O^*_Y)\to H^0(Y,v_*(\mathcal O^*_C))\to \\\to H^0(Y,\mathcal G)\to H^1(Y,\mathcal O^*_Y)\to H^1(Y,v_*(\mathcal O^*_C)) \to 1
\end{multline*}

\noindent which is given as the exact sequence
$$1\to k^*\to(k^*)^{\gamma_Y}\to(k^*)^{\delta_Y}\to\Pic(Y)\to\Pic(C) \to 0.$$
Therefore we obtain the following exact sequence 
$$1\to(k^*)^{b_1(\Gamma_Y)}\to\Pic(Y)\stackrel{v^*}\to\Pic(C)=\prod_{i=1}^r \Pic(C_i)\to 0$$
where $v^*$ is the pullback of line bundles and $C_i$ is the normalization of $Y_i$.

Now let $Y=\coprod\limits_{i=1}^m Y_i$ be the decomposition of $Y$ into its irreducible components $Y_i$, let $\underline d=(d_1,...,d_m)\in \Bbb Z^ m$, and define 
$$\text{Pic}^{\underline d}(Y):=\{L\in\text{Pic}(Y):\text{deg}(L|_{Y_i})=d_i\}.$$ 
This is the locus parametrizing line bundles of multidegree $\underline d$ on $Y$. Following Caporaso \cite{Ca1, Ca2} we obtain that for every $d\in\Bbb Z$, there exists a Picard scheme $P^d_Y$ parametrizing line bundles of a suitable degree $d$ and we have
$$P^d_Y=\coprod\limits_{\underline d\in\Delta_d,\hskip1mm|\underline d|=d}\text{Pic}^{\underline d}(Y),$$ 
\noindent for a well-defined finite set $\Delta_d$ and where $|\underline d|:=\sum_{i=1}^m d_i$. For a nodal algebraic curve $Y$ of compact type with irreducible components $Y_1,...,Y_m$ we get the following decomposition isomorphism
$$\text{Pic}^{\underline d}(Y)\xrightarrow{\simeq}\text{Pic}^{d_1}(Y_1)\times\cdots\times\text{Pic}^{d_m}(Y_m),\hskip1mm L\to(L|_{Y_1},\dots,L|_{Y_m}).$$
In other words, specifying a line bundle on a nodal curve of compact type is given by fixing a multidegree $\underline d=(d_1,...,d_m)\in \Bbb Z^ m$
and choosing a line bundle of degree $d_i$ on each component $Y_i$ of the nodal curve for every $i=1, 2,\ldots, m$.

Every line bundle can be also interpreted as a principal $\C^*$-bundle. These are classified by the classifying stack $\Bb \C^*$ of the algebraic group $\C^*$. Its cohomology resembles the singular cohomology of the classifying space $BU(1)=\C P^{\infty}$, which plays a similar role in algebraic topology. In fact, we have the following isomorphism (see also \cite{He, N})

\begin{theorem}\label{thm:classifyingstack}
The rational cohomology algebra of the classifying stack $\Bb \C^*$ of principal $\C^*$-bundles is given as the graded $\Q$-algebra
$$H^*(\Bb \C^*, \Q)\cong \Q[c_1],$$
with the generator $c_1$ being the Chern class of degree $2$ of the universal line bundle $\mathcal L^{univ}$ on $\Bb \C^*$.
\end{theorem}

\begin{proof} The quotient morphism $\pi: \A^n-\{0\}\rightarrow \P^{n-1}$ is a principal $\C^*$-bundle and so we have a cartesian diagram of algebraic stacks with classifying morphism $\phi$
\[
\xy \xymatrix{ \A^n-\{0\} \ar[r] \ar[d]^{\pi} & \Spec(\C) \ar[d]\\
 \P^{n-1}\ar[r]^{\phi} & \Bb \C^*}
\endxy
\]
The Leray spectral sequence for $\phi$ is of the form
$$E^{s,t}_2\cong H^s(\Bb \C^*, R^t\phi_*\Q)\Rightarrow H^*(\P^{n-1}, \Q)$$
and because $R^0\phi_*\Q\cong \Q$ and $R^t\phi_*\Q=0$ if $t< 2n-1$ it follows for $t< 2n-1$ that
$$H^t(\Bb \C^*, \Q)\cong H^t(\P^{n-1}, \Q)$$ 
and therefore 
$$H^*(\Bb \C^*, \Q)\cong \Q[c_1]$$
where $c_1$ is the Chern class of the universal bundle $\L^{univ}$ on the classifying stack $\Bb \C^*$.
\end{proof}

Now let $X$ be a smooth projective irreducible algebraic curve of genus $g\geq 2$ and consider the moduli stack $\Bun_X^{1, d}$ of line bundles of degree $d$ on $X$. A course moduli space for the algebraic stack $\Bun_X^{1, d}$ is given by the Picard scheme 
$\Pic^d_X$ of $X$. This assembles into a $\C^*$-gerbe 
$$\Bun_X^{1, d}\rightarrow \Pic^{d}_X.$$
Furthermore, there exists a Poincar\'e universal family on the product scheme $X\times \Pic_X^d$ which gives a section of the $\C^*$-gerbe and therefore a splitting isomorphism (see \cite{Ra}, \cite[Thm. G]{DrNa} and \cite{He})  of algebraic stacks
$$\Bun_X^{1, d}\cong \Pic^{d}_X \times \Bb\C^*$$
We can now use the K\"unneth theorem to determine the cohomology of the moduli stack $\Bun_X^{1, d}$ and obtain
$$H^*(\Bun_X^{1,d}, \Q)\cong H^*(\Pic^{d}_X, \Q)\otimes H^*(\Bb \C^*, \Q).$$
\noindent The cohomology of the Picard scheme $\Pic^d_X$ is isomorphic to the cohomology of the Jacobian $\Jac(X)$ of the algebraic curve $X$, but the Jacobian is an abelian variety and we have
$$H^*(\Jac(X), \Q)\cong \Lambda_{\Q} (H^1(X, \Q))\cong \Lambda_{\Q}(\alpha_1, \ldots \alpha_{2g}).$$

 Therefore we have determined the rational cohomology algebra of the moduli stack of line bundles on a smooth algebraic curve as follows (compare also \cite{He, HeSch, N})

\begin{theorem}\label{thm:smooth}
Let $X$ be a smooth projective irreducible algebraic curve of genus $g\geq 2$. The cohomology of the moduli stack $\Bun_X^{1, d}$ of line bundles of degree $d$ over $X$ is given as the graded $\Q$-algebra 
$$H^*(\Bun_X^{1,d}, \Q)\cong \Lambda_{\Q}(\alpha_1, \ldots \alpha_{2g})\otimes \Q[c_1].$$
\end{theorem}

We will now discuss the more general case of a nodal curve of compact type. For this let us define (compare \cite{Ho})

\begin{definition}
Let $Y$ be connected nodal algebraic curve of compact type over $k$ with smooth irreducible components $Y_1,...,Y_m$ and  $\Bun_Y^{n, d}$ be the moduli stack of vector bundles of rank $n$ and degree $d$ over $Y$. We say that a vector bundle $\mathcal F$ over a dense open substack $\mathscr U\subseteq\Bun_Y^{n,d}$ has {\em weight} $w\in\Bbb Z$ if the diagram 

$$
\xymatrix{
k^*\ar[d]^{(-)^w}&\hspace*{-0.7cm}\longhookrightarrow \text{Aut}_{\mathcal O_Y}(E) \ar[d] \\
k^* &\hspace*{-0.7cm}\longhookrightarrow       \text{Aut}_k(\mathcal F_E) }
$$

\noindent commutes for all vector bundles $E$ over $Y$ that are objects of the groupoid $\mathscr U(k)$ of $k$-rational points of 
$\mathscr U$ and where $\mathcal F_E$ is the associated vector space over $k$.
\end{definition}

\begin{example}\label{ex:sections} Denote by $\mathcal E^{\text{univ}}$ the universal vector bundle of rank $n$ and degree $d$ over the algebraic stack $Y\times\Bun_Y^{n,d}$, and for a given rational point $p\in Y(k)$ denote by $\mathcal E^{\text{univ}}_p$ its restriction to $\{p\}\times \Bun_Y^{n,d}\cong\Bun_Y^{n,d}$. We see that $\mathcal E^{\text{univ}}_p$ is a vector bundle of weight $1$ on $\Bun_Y^{n,d}$ and its dual is a vector bundle of weight $-1$. 
\end{example}

The case of line bundles is particularly important as it provides general criteria for the splitting of $\C^*$-gerbes, namely we have
\cite[Lemma 3.10]{He}

\begin{lemma}\label{lem:sections} For a $\C^*$-gerbe of algebraic stacks $\Phi: \mathscr{X}\rightarrow \mathscr{Y}$ the following are equivalent conditions:
\begin{itemize}
\item[(i)] The $\C^*$-gerbe of algebraic stacks $\Phi: \mathscr{X}\rightarrow \mathscr{Y}$ has a section.
\item[(ii)] There is an isomorphism of algebraic stacks $\mathscr{X} \cong \Bb\C^*\times \mathscr{Y}$.
\item[(iii)] There exists a line bundle of weight 1 on the algebraic stack $\mathscr{X}$.
\end{itemize}
A $\C^*$-gerbe satisfying any of the above equivalent conditions is also called {\em neutral}.
\end{lemma}

With these preparations, we can now determine the rational cohomology algebra of the moduli stack of line bundles of degree $d$ over a nodal algebraic curve $Y$ of compact type.

\begin{theorem} Let $Y$ be a nodal algebraic curve of compact type with irreducible smooth components $Y_1,...,Y_m$. 
The cohomology of the moduli stack $\Bun_Y^{1, d}$ of line bundles of degree $d$ over $Y$ is the graded $\Q$-algebra
$$H^*(\Bun_Y^{1,d},\Q)\cong [\bigoplus\limits_{i=1}^m \Lambda_{\Q}(\alpha^{(i)}_1, \ldots \alpha^{(i)}_{2g_i})]\otimes \Q[c_1],$$
where the $\alpha^{(i)}_j$ for $i=1,\ldots m;\; j=1, \ldots 2g_i$ are generators corresponding to the first cohomology classes of the irreducible smooth components $Y_i$ of $Y$ and $c_1$ corresponds to the first Chern class of the universal bundle $\L^{univ}$ on the classifying stack $\Bb \C^*$.
\end{theorem}

\begin{proof} We have a $\C^*$-gerbe of algebraic stacks $$\Bun_Y^{1, d}\rightarrow \P_Y^d$$
and using the universal line bundle $\mathcal L$ of degree $d$ on the algebraic stack $Y\times \Bun_Y^{1,d}$ we get from Example \ref{ex:sections} and Lemma \ref{lem:sections} a splitting of algebraic stacks
$$\Bun_Y^{1,d}\cong P^d_Y\times \Bb \C^*$$
since we have the decomposition
$$P^d_Y=\coprod\limits_{\underline d\in\Delta_d,\hskip1mm|\underline d|=d}\Pic^{\underline d}(Y),$$ 
where $$\Pic^{\underline d}(Y)\xrightarrow{\simeq}\Pic^{d_1}(Y_1)\times\cdots\times\Pic^{d_m}(Y_m).$$ 
For each irreducible smooth component $Y_i$ being of a genus $g_i\geq 2$ we get thus from Theorem \ref{thm:smooth}

$$
\begin{aligned}
H^*(\Bun_{Y_i}^{1,d_i},\Q)&\cong H^*(\Pic^{d_i}(Y_i), \Q)\otimes H^*(\Bb \C^*, \Q)\\
&\cong \Lambda_{\Q}(\alpha^{(i)}_1, \ldots \alpha^{(i)}_{2g_i})\otimes \Q[c_1]
\end{aligned}
$$
Using again the K\"unneth decomposition for rational cohomology we can hence determine the rational cohomology algebra of the moduli stack 
$\Bun_Y^{1,d}$ of line bundles on the nodal curve $Y$ and therefore get from the discussions above

$$
\begin{aligned}
H^*(\Bun_Y^{1,d},\Q) &\cong [(\bigoplus\limits_{i=1}^m H^*(\Pic^{d_i}(Y_i),\Q))]\otimes H^*(\Bb \C^*, \Q)&\\
&\cong\bigoplus\limits_{i=1}^m[H^*(\text{Pic}^{d_i}(Y_i), \Q)\otimes H^*(\Bb \C^*, \Q)]&\\
& \cong \bigoplus\limits_{i=1}^m H^*(\Bun_{Y_i}^{1,d_i}, \Q). &\\
& \cong [\bigoplus\limits_{i=1}^m \Lambda_{\Q}(\alpha^{(i)}_1, \ldots \alpha^{(i)}_{2g_i})]\otimes \Q[c_1]
\end{aligned}
$$
\noindent with the desired generators of the rational cohomology algebra.
\end{proof} 

\begin{remark}
Determining the rational cohomology algebra of moduli stacks of general principal $G$-bundles or vector bundles over a fixed nodal algebraic curve is much harder. The moduli stacks of principal $G$-bundles over nodal curves don't behave as good as moduli stacks of line bundles, for example, they are not complete. It will be more convenient to work instead with moduli stacks of Gieseker bundles or torsion free sheaves over nodal curves (compare \cite{Ka}). These moduli stacks and their cohomology will be discussed in a follow-up article by the authors.
\end{remark}

\section*{Acknowledgements}
\noindent The first author is supported by Grant PAPIIT UNAM IN100723, ``Curvas, sistemas lineales en superficies proyectivas y fibrados vectoriales".  The second author would like to thank Centro de Ciencias Matem\'aticas (CCM), UNAM Campus Morelia, for the wonderful hospitality and financial support through the Programa de Estancias de Investigaci\'on (PREI) de la Direcci\'on General Asuntos del Personal Acad\'emico, DGAPA-UNAM. He also likes to thank the organizers of the XIII Annual International Conference of the Georgian Mathematical Union in Batumi for the kind invitation and great hospitality. Both authors thank the referees for their valuable comments.

\end{document}